\documentclass[10pt,reqno]{amsart}
\usepackage{amsmath}
\usepackage{amssymb}
\usepackage{amsthm}
\usepackage{enumitem}
\usepackage{romannum}
\newlength{\defbaselineskip}
\newcommand{\setlinespacing}[1]%
           {\setlength{\baselineskip}{#1 \defbaselineskip}}

\numberwithin{equation}{section}

\newtheorem{thm}{Theorem}[section]

\newtheorem{lem}[thm]{Lemma}
\newtheorem{prop}[thm]{Proposition}

\theoremstyle{definition}

\theoremstyle{remark}
\newtheorem{rem}[thm]{Remark}
\numberwithin{equation}{section}

\begin{document}
\pagenumbering{arabic}\setcounter{page}{1}

\title[Dispersive quantization and fractalization]
{On dispersive quantization and fractalization for the Kawahara equation}

\author{Seongyeon Kim}

\thanks{This work was supported by the POSCO Science Fellowship of POSCO TJ Park Foundation.}


\address{Department of Mathematics Education, Jeonju University, Jeonju 55069, Republic of Korea}
\email{sy\_kim@jj.ac.kr}

\begin{abstract}
In this paper, we investigate the dichotomous behavior of solutions to the Kawahara equation with bounded variation initial data, analogous to the Talbot effect. Specifically, we observe that the solution is quantized at rational times, whereas at irrational times, it is a nowhere continuous differentiable function with a fractal profile. This phenomenon, however, has not been explored for the Kawahara equation, which is a fifth-order KdV type equation. To achieve this, we derive smoothing estimates for the nonlinear Duhamel solution, which, when combined with the known results on the linear solution, provides a mathematical description of the Talbot effect.
\end{abstract}

\maketitle

\section{Introduction}\label{sec1}
In 1836, Talbot \cite{T1} made an experimental discovery that a plane wave, when transmitted through a periodic grating, produces a diffracted wave whose image is refocused to reproduce the original grating pattern at certain intervals. Subsequently, the interval, now referred to as the Talbot distance $d_T$, was determined by Rayleigh \cite{R1} to be $d_T = a^2 / \lambda$, where $a$ represents the grating's spacing and $\lambda$ the wavelength of the incoming wave.

The Talbot effect has been analyzed mathematically through the study of \textit{dispersive quantization} and \textit{fractalization} in solutions to linear dispersive PDEs.
In their work \cite{Be,BMS}, Berry and his collaborators demonstrated that with a step function as initial data, the solution to the linear Schrödinger equation manifests as piecewise constant yet exhibits discontinuities at $t\in\pi\mathbb{Q}$. Conversely, at times $t\not\in\pi\mathbb{Q}$, the solution remains continuous but becomes nowhere differentiable and develops a fractal-like pattern.
Further, Berry and Klein \cite{BK} elucidated the Talbot effect mathematically by showing that at each time $t\in d_T\mathbb{Q}$, the grating pattern is replicated a finite number of times, whereas at times $t\not\in d_T\mathbb{Q}$, the resulting images display a fractal and nowhere differentiable profile.

Moreover, numerous researchers have demonstrated that the pronounced dichotomy between quantization and fractalization at rational versus irrational times is present in linear dispersive equations featuring polynomial dispersion relations with integer coefficients on the torus.
Assuming that the initial data is of bounded variation (BV), Oskolkov \cite{Os} demonstrated that if the initial data includes discontinuities, then at $t\in\pi\mathbb{Q}$, the solution to any linear dispersive PDE on $\mathbb{T}$ (including the linear Schrödinger and Airy equations) remains a bounded function but may have a finite or countably infinite number of discontinuities.
Conversely, at times $t\not\in\pi\mathbb{Q}$, the solution becomes continuous yet nowhere differentiable.
Kapitanski and Rodnianski \cite{KR} subsequently found that the regularity of solutions to the linear Schrödinger equation improves at $t\not\in\pi\mathbb{Q}$.
Taylor \cite{Tay} specifically detailed the quantization effect at $t\in\pi\mathbb{Q}$ by illustrating that when $g$ is the delta function, the solution $e^{it\partial_{xx}}g$ transforms into a finite linear combination of shifted delta functions, with coefficients composed of Gauss sums.
Olver \cite{Ol} expanded this finding to include arbitrary order linear dispersive PDEs with integer coefficient polynomial dispersion relations.

An alternative method to mathematically analyze the Talbot effect involves studying the fractal dimension\footnote{The \emph{fractal dimension}, also referred to as \emph{upper Minkowski dimension}, of a set $S\subset \mathbb R^n$ is defined by $$\limsup_{\epsilon\to 0} \frac{\log(N(S,\epsilon))}{\log(1/\epsilon)},$$ where $N(S,\epsilon)$ is the minimum number of balls of radius $\epsilon$ required to cover $S$.} of the graph of solutions, which describes the fractalization behavior.
Specifically, Berry \cite{Be} proposed that for the $n$-dimensional linear Schrödinger equation, the graphs of $\Re u(x,t)$, $\Im u(x,t)$, and $|u(x,t)|^2$ exhibit a fractal dimension of $D=n+\frac{1}{2}$ at almost all irrational times $t$.
In the one-dimensional case, Rodnianski \cite{R2} confirmed that for any initial data $g\in \textnormal{BV}$ but not belonging to $\cup_{\epsilon>0} H^{\frac12+\epsilon}$, the graphs of the real and imaginary parts of the solution $e^{it\partial_{xx}}g$ display a fractal dimension of $D=\frac32$ at almost all irrational times, offering partial support for Berry's conjecture.
Subsequently, Chousionis, Erdo\u{g}an, and Tzirakis in \cite{CET} broadened this finding to encompass general linear dispersive equations with integer coefficient polynomial dispersion of any degree $d$ by considering initial data within the broader class $\textnormal{BV}\setminus\bigcup_{\epsilon >0}H^{\sigma_0+\epsilon}$, where $\frac{1}{2}\leq \sigma_0<\frac12+\frac1{2^d}$ (refer also to \cite{ET-book}).
Moreover, they resolved Berry's conjecture in one dimension by demonstrating that the graph of $|e^{it\partial_{xx}}g|^2$ measures a fractal dimension of $3/2$ for almost all $t\notin \pi\mathbb Q$, when $g$ is a step function with jumps solely at rational points.

Recently, it has been revealed that nonlinear dispersive equations also demonstrate the intriguing phenomena where the solution becomes quantized at $t\in\pi\mathbb{Q}$, and at $t\not\in\pi\mathbb{Q}$, it transforms into a continuous yet nowhere differentiable function with a fractal profile.
The studies documented in \cite{ET1,ET,CET,ES,ch-ol14} present findings within the nonlinear Schrödinger and Korteweg-de Vries (KdV) equations.
A crucial step in analyzing nonlinearity involves acquiring smoothing estimates for the nonlinear component. When these are integrated with existing findings about the linear component (refer to \cite{Os,R2,ET-book, ckskk} and related references), they contribute to understanding both the quantization and fractalization phenomena, as well as the fractal dimension.
Using this analytical framework, we explore the quantized and fractal behaviors of solutions to the Kawahara equation on the torus, a subject previously unstudied:
\begin{equation}\label{KA}
	\begin{aligned}
		&\partial_t u + \partial_x^5 u+\alpha \partial_x^3 u+ u\partial_xu=0,\ \ x\in\mathbb T=\mathbb R/(2\pi\mathbb{Z}),\ \,t\ge0,\\
		&u(x,0)=g(x)\in H^s(\mathbb T),
	\end{aligned}
\end{equation}
where $\alpha$ can be $-1, 0$ or $1$ and $u$ is a real-valued function.
This fifth-order KdV type equation is used to model phenomena such as shallow water waves with surface tension and magneto-sound propagation in plasma (refer to \cite{HS,Kaw,KS}).

Our first result concern the dichotomy between rational and irrational times.
\begin{thm}\label{Ta}
Let $u$ be the solution to \eqref{KA} with $g\in \text{BV}$. 
If $t \not\in \pi \mathbb{Q}$, then $u(x,t)$ is a continuous function of $x$. 
If $t \in \pi \mathbb{Q}$ and $g$ has at least one discontinuity on $\mathbb{T}$, then $u(x,t)$ is a bounded function with at most countably many discontinuity.
If $g$ is continuous, the $u(x,t)\in C_t^0C_x^0$.
\end{thm}

We further calculate the fractal dimension of the solution's graph at irrational time slices based on the regularity of the initial data.
\begin{thm}\label{Fr}
Let $u$ be a solution to \eqref{KA} with $g\in \textnormal{BV}$.
Suppose that\footnote{Since $g\in BV$ it follows that $\sigma_0\ge1/2$.}
\begin{equation*}
\sigma_0 := \sup \{\sigma \in \mathbb{R} : g \in H^{\sigma}\}<\frac{17}{32}.
\end{equation*}
Then, for almost all $t\in\mathbb{R}\setminus \pi \mathbb{Q}$,
the graph of the solution has upper Minkowski dimension $D\in[\frac{33}{16}-2\sigma_0,\frac{31}{16}]$.
Moreover, the graph of $|u(x,t)|^2$ has upper Minkowski dimension at most $\frac{31}{16}$. 
\end{thm}

We end the introduction with some notations.
We define the Fourier coefficient of a $2\pi$-periodic function $f$ as $\widehat f(k)=\frac1{2\pi}\int_0^{2\pi} e^{-ikx}f(x) dx$.
For every $s\ge 0$, we denote the Sobolev space $H^s$ on $\mathbb T=\mathbb R/2\pi \mathbb Z$, which is equipped with the norm
\begin{equation*}
	\|f\|_{H^s} = \Big( \sum_{k\in\mathbb Z} \langle k\rangle^{2s} |\widehat f (k)|^2 \Big)^{1/2}
\end{equation*}
where $\langle k\rangle=(1+|k|^2)^{1/2}$.
Throughout this paper, we use the following notation if a Banach space of functions $\mathcal B^s$ decreases (in terms of set inclusion) with respect to a regularity index $s\in\mathbb R$:
\begin{equation*}
	\mathcal B^{s+} = \bigcup_{\epsilon>0} \mathcal B^{s+\epsilon} , \quad \mathcal B^{s-} = \bigcap_{\epsilon>0} \mathcal B^{s-\epsilon}.
\end{equation*}
Specifically, $\mathcal B^+:=\mathcal B^{0+}$.
For instance, $\mathcal B^{s}$ includes the Sobolev space $H^s$, the Besov space $B_{p,\infty}^s$, and the H\"older space $C^s$.
The letter C represents a positive constant which may differ in each instance. We also use $A\lesssim B$ to indicate $A\le CB$ with constants $C>0$ unspecified.

\

\noindent\textit{Outline of the paper.}
In Section \ref{sec2} we state the smoothing estimate (Proposition \ref{thm1}), which is crucially used in proofs of Theorems \ref{Ta} and \ref{Fr}. We then prove Theorems \ref{Ta} and \ref{Fr} by looking at linear and non-linear evolutions of the solution separately; the non-linear evolution lies in a smoother space due to the nonlinear smoothing property (Proposition \ref{thm1}), which, when combined with the known results on the linear evolution.
In Section \ref{sec3}, we establish local well-posedness for \eqref{KA} and prove the smoothing estimate in Section \ref{sec4}.
To this end, we make use of the mean-zero property of the initial data and the momentum conservation. We replace $u$ with $v=u-\int_\mathbb{T}g(x)dx$ using a harmless first order term $[\int_\mathbb{T}g(x)dx]\partial_xv$ so that the assumption on the initial data is satisfied. 
We note that this assumption can be removed by changing the linear operator in \eqref{KA} from $- \partial_x^5 - \alpha \partial_x^3$ to $- \partial_x^5 - \alpha \partial_x^3 - [\int_{\mathbb{T}}g(x)dx] \partial_x$.
We observe that after this replacement, nothing changes in all the calculations when adapting the methods used in the proofs. This observation was introduced by Bourgain in \cite{B} (see also \cite{CKSTT1,ET}).

\

\section{Talbot effect for the Kawahara equation}\label{sec2}
In this section we prove Theorems \ref{Ta} and \ref{Fr}.
Let us first formally decompose the solution to \eqref{KA} as 
\begin{equation}\label{s}
u(x,t) = e^{Lt}g + \mathcal N(x,t)
\end{equation}
where $L=-\partial_x^5 -\alpha \partial_x^3 -[\int_{\mathbb{T}}g(x)dx]\partial_x$ and  $\mathcal N(x,t)$ is the nonlinear Duhamel term of the solution. 
Here, the first order term in $L$ is arisen in removing the mean-zero assumption, and does not affect in the proofs at all, as mentioned in Section \ref{sec1}.
 
The key ingredient of the proofs is the following smoothing property for $\mathcal{N}(x,t)$, which will be shown in Section \ref{sec4}:
\begin{prop}\label{thm1}
Let $s>-1/2$, $0\leq s_1< \min(s+2,3s+2)$.
Let $u$ be the solution to \eqref{KA}.
Then we have  $$\mathcal N(x,t)\in C_t^0H_x^{s_1}.$$
\end{prop}

\subsection{Proof of Theorem \ref{Ta}}
Since $g\in BV$ we see $g \in H^{\frac12-}$.
Then, by making use of the Proposition \ref{thm1}, we have 
$$ \mathcal{N}(x,t) \in C_t^0 H_x^{\frac52-}.$$
From the Sobolev embedding 
\begin{equation}\label{emb}
	H^s \hookrightarrow C^{s-\frac{1}{2}} \quad \textrm{for} \quad s>\frac{1}{2},
\end{equation}
we also have 
\begin{equation}\label{c}
	\mathcal{N}(x,t) \in C_t^0 C_x^{2-}.
\end{equation}

To address the linear solution $e^{Lt}g$, we utilize the following established result due to Oskolkov \cite{Os}:
\begin{thm}\label{1}
Let $L=-\partial_x^5 -\alpha \partial_x^3 -[\int_\mathbb{T}g(x)dx] \partial_x$ and 
suppose that $g\in BV$.
\begin{enumerate}
[leftmargin=0.6cm]
\item[$(\romannum{1})$] If $t\not \in \pi\mathbb{Q}$, then $e^{Lt}g$ is a continuous function of $x$. 
If $t\in\pi\mathbb{Q}$ and $g$ has at least one discontinuity on $\mathbb{T}$, then $e^{Lt}g$ is a bounded function with at most countably many discontinuities.
\item[$(\romannum{2})$] If $g$ is continuous, then $e^{Lt}g\in C_t^0 C_x^0$.
\end{enumerate}
\end{thm}
\begin{rem}\label{r}
In fact, it is known that the theorem holds for one dimensional dispersive equations with polynomial dispersion relation in \cite{Os}.
The quantization results in \cite{BK,Tay,Ol} says that the linear solution $e^{Lt}g$ is a linear combination of finitely many translates of the initial data $g$ when $t\in\pi\mathbb{Q}$ (see \cite[Theorem 2.14]{ET-book}).
Thus, when $t\in\pi\mathbb{Q}$ and $g\in BV$ has at least one discontinuity, $e^{Lt}g\in BV$ and the possible discontinuities are at most countable.
\end{rem}

For $t\not\in\pi\mathbb{Q}$, combining Theorem \ref{1} $(\romannum{1})$ and \eqref{c} with \eqref{s} we conclude that $u(x,t)$ is also continuous.
If $t\in\pi\mathbb{Q}$ and $g$ is discontinuous, then $e^{Lt}g$ is of bounded variation with at most countably many discontinuities by Theorem \ref{1} $(\romannum{1})$ and Remark \ref{r}.
Combining this with \eqref{c}, we conclude that the solution is a discontinuous bounded function with at most countable discontinuity.

If $g$ is continuous, then $e^{Lt}g$ is also continuous on $\mathbb{T} \times \mathbb{R}$ by Theorem \ref{1} $(\romannum{2})$. Therefore, combining with \eqref{c} we see that $u(x,t)\in C_t^0C_x^0$.

\subsection{Proof of Theorem \ref{Fr}}
We begin by introducing the Besov space and its properties that we need. 
Let $\phi \in C^\infty_0([-2,-1/2] \cup [1/2,2])$ be such that $\sum_{j\in\mathbb Z}\phi(2^{-j}t)=1$ for $t\in\mathbb R\setminus\{0\}$, and let $\phi_0(t):=1-\sum_{j\ge1}\phi(2^{-j}t)$. We denote by $P_j$ the projections defined by  
\[	P_0f(x):=\sum_{k\in \mathbb Z} \phi_0(k)\widehat f(k)e^{ikx}, \quad P_j f(x):= \sum_{k\in \mathbb Z} \phi(2^{-j}k)\widehat f(k)e^{ikx}, \ \ j\ge 1.	\] 
For $1\le p\le\infty$ and $s\ge 0$, the Besov space $B^s_{p,\infty}$ on the periodic domain $\mathbb T$ is defined via the norm 
\begin{equation*}
\|f\|_{B_{p,\infty}^s} := 
\sup_{j\ge 0} 2^{sj}\|P_jf\|_{L^p(\mathbb{T})}.
\end{equation*}
It is well-known that 
$B^\alpha_{\infty, \infty} = C^\alpha$ for $\alpha\in (0,\infty)\setminus \mathbb N$ and 
$H^{\alpha}\in B_{1,\infty}^\alpha$ for $\alpha\ge0$ (see e.g., \cite{ST87,BL}).
By complex interpolation between Besov spaces (see \cite[Theorem 6.4.5]{BL}) and H\"older's inequality, we have, for $s_1\neq s_2$,
\begin{equation}\label{e:inter}
	\big(B_{1,\infty}^{s_1},B_{\infty,\infty}^{s_2} \big)_{[\frac{1}{2}]}=B_{2,\infty}^{\frac{s_1+s_2}{2}} \hookrightarrow B_{2,2}^{\frac{s_1+s_2}2 -} = H^{\frac{s_1+s_2}{2} -}.	
\end{equation}

Now we state the following theorem of Chousionis--Erdo\u{g}an--Tzirakis \cite{CET} (see also \cite{ET-book}) which is useful to prove Theorem \ref{Fr}.
\begin{thm} \label{Linear}
Let $\frac{1}{2}\leq \sigma_0 < \frac{17}{32}$. If $g\in BV\setminus H^{\sigma_0+}$, then for almost all $t\in\mathbb{R}\setminus \pi \mathbb{Q}$, we have $e^{Lt}g \in C^{\frac{1}{16}-} \setminus B_{1,\infty}^{2\sigma_0-\frac{1}{16}+}$.
\end{thm}

We initially prove the upper bound of $D$ is $\frac{31}{16}$.
Given that $g$ belongs to $H^{\sigma_0 -}$ as defined by $\sigma_0$, we derive
\begin{equation}\label{smo1}
	\mathcal N(\cdot,t)\in H^{\sigma_0+2-},\quad \forall  t\in\mathbb{R}
\end{equation}
from Proposition \ref{thm1}.
Therefore, according to the Sobolev embedding \eqref{emb}, we obtain
\begin{equation}\label{smo}
	\mathcal N(\cdot,t)\in C^{\sigma_0+\frac32-}, \quad \forall t\in\mathbb{R}.
\end{equation}
On the other hand, by Theorem \ref{Linear}, we have 
\begin{equation}\label{sm}
e^{Lt}g\in C^{\frac1{16}-} \quad \textrm{for almost all}\,\, t\in\mathbb{R}\setminus \pi \mathbb{Q}. 
\end{equation}
Hence, from \eqref{s},\eqref{smo} and \eqref{sm} we have 
\begin{equation}\label{2}
u(\cdot,t)\in C^{\min(\sigma_0+\frac32,\frac1{16})-}=C^{\frac1{16}-} \quad \textrm{for almost all}\,\, t\in\mathbb{R}\setminus \pi \mathbb{Q}.
\end{equation}

We now use the following classical result on the upper Minkowski dimension for H\"older continuous functions:
\begin{lem}\cite[Corollary 11.2]{F}\label{3}
Let $0\leq \alpha \leq 1$. If a function $f:\mathbb{T} \rightarrow\mathbb{R}$ is in $C^{\alpha}$, then the upper Minkowski dimension of the graph of $f$ is at most $2-\alpha$.
\end{lem}
Indeed, applying \eqref{2} to Lemma \ref{3}, the upper Minkowski dimension of the graph of $u(x,t)$  is at most $2-1/{16}={31}/{16}$.
Just by replacing $u(x,t)$ with $|u(x,t)|^2$ in the argument employed for obtaining \eqref{2}, it follows that the upper bound of $D$ is at most $31/16$.

It remains to prove that the lower bound of $D$ is $\frac{33}{16}-2\sigma_0$. 
The lower bound follows from the following theorem of Deliu--Jawerth \cite{DJ}: 
\begin{thm}[\cite{DJ}]\label{FDL}
Let $0<s<1$. If $f\colon\mathbb{T}\to \mathbb{R}$ is continuous and $f\notin B_{1,\infty}^{s+}$, then the upper Minkowski dimension of the graph of $f$ is at least  $2-s$.
\end{thm}

Since $g\not\in H^{\sigma_0+}$ and $\|g\|_{H^{s}}=\|e^{Lt}g\|_{H^s}$, one can easily see that $e^{Lt}g \not \in {H^{\sigma_0+}}$.
From \eqref{sm} with $C^{\frac1{16}-}=B_{\infty,\infty}^{\frac1{16}-}$ and the embedding $	B_{1,\infty}^{s_1}\cap B_{\infty,\infty}^{s_2} \subset H^{\frac{s_1+s_2}{2} -}$ (see \eqref{e:inter}), we have 
\begin{equation}\label{11}
e^{Lt}g\not\in B_{1,\infty}^{2\sigma_0-\frac{1}{16}+}
\end{equation}
for almost all $t\in\mathbb{R}\setminus\pi\mathbb{Q}$.
On the other hand, by \eqref{smo1} with $H^{\sigma_0+2-}\in B_{1,\infty}^{\sigma_0+2-}$, we see, for all $t\in\mathbb{R}$, 
\begin{equation}\label{12}
\mathcal N(\cdot,t)\in B_{1,\infty}^{\sigma_0+2-}.
\end{equation} 

Combining \eqref{11} and \eqref{12}, then we have
$$u(\cdot,t)= \underbrace{e^{Lt}g}_{\notin B^{2\sigma_0-\frac{1}{16}+}_{1,\infty}\text{ for a.e. }t} 
+\underbrace{\mathcal N(\cdot,t)}_{\in B^{\sigma_0+2-}_{1,\infty} \text{ for all }t}.$$
Hence we conclude that if $2\sigma_0 -\frac{1}{16}<\sigma_0 +2$, then for almost all $t\in\mathbb{R}\setminus \pi\mathbb{Q}$ the solution does not belong to $B_{1,\infty}^{2\sigma_0 -\frac{1}{16}+}$.
Therefore, by Theorem \ref{FDL}, the graph of the solution has Minkowski dimension at least $2-(2\sigma_0-\frac{1}{16})=\frac{33}{16}-2\sigma_0$.

\section{Local well-posedness}\label{sec3}
In this section, we establish the local well-posedness of the Cauchy problem \eqref{KA}.
We begin by revisiting established results concerning the Cauchy problem \eqref{KA}.
In \cite{H}, Hirayama demonstrated that the solution is locally well-posed in $\dot{H}^s(\mathbb T)$ for $s\ge-1$ by making use of the argument employed for the KdV equation in \cite{KPV} .
Kato \cite{Kat2} enhanced the result to $s\ge-3/2$, and proved global well-posedness for $s\ge-1$ as well as $C^3$-ill-posedness for $s<-3/2$.

Before presenting our well-posedness result, we need to define the appropriate function spaces.
Let $X^{s,b}$ denote the Bourgain space under the norm
$$\|u\|_{X^{s,b}}:=\|\langle k\rangle^s\langle\tau+ k^5-\alpha k^3\rangle^b\widehat u(k,\tau)\|_{l_k^2L_\tau^2}.$$
We further define, for any $\delta\ge0$, the restricted space $X_{\delta}^{s,b}$ using the norm
$$\|u\|_{X_\delta^{s,b}}:=\inf_{\tilde u=u\ \text{on}\ t\in[0,\delta]}\|\widetilde u\|_{X^{s,b}}.$$
Additionally, we introduce the $Y^s$- and $Z^s$-spaces, inspired by the idea of Bourgain \cite{B} (cf. \cite{GTV,CKSTT1}),
\begin{align*}
	&\|u\|_{Y^s}:=\|u\|_{X^{s,\frac12}}+\|\langle k\rangle^s\widehat u(k,\tau)\|_{l_k^2L_\tau^1},\\
	&\|u\|_{Z^s}:=\|u\|_{X^{s,-\frac12}}+\bigg\|\frac{\langle k\rangle^s\widehat u(k,\tau)}{\langle\tau+k^5-\alpha k^3\rangle}\bigg\|_{l_k^2L_\tau^1}.
\end{align*}
Similarly, we specify $Y_\delta^s$, $Z_\delta^s$.
We note that if $u\in Y^s$ then $u\in L_t^\infty H_x^s$ (indeed, $Y^s$  is embedded into $C_tH_x^s$); however the $X^{s,\frac12}$-norm alone does not suffice to control the $L_t^\infty H_x^s$-norm of the solution.

The following theorem shows the local well-posedness of \eqref{KA} in $H^s$ with $s\ge0$.
\begin{thm}\label{thm2}
The Cauchy problem \eqref{KA} is locally well-posed in $H^s$ for any $s\ge0$. Namely, there exist $\delta>0$ and a unique solution $u\in C([0,\delta];H^s(\mathbb T))\cap Y_\delta^s$ with
	\begin{equation}\label{bd}
	\|u\|_{X_\delta^{s,\frac12}}\le\|u\|_{Y_\delta^s}\lesssim\|g\|_{H^s}.
	\end{equation}
\end{thm}

\begin{proof}[Proof of Theorem \ref{thm2}]
We may assume that the initial data $g(x)$ has the mean-zero property, as discussed in Section \ref{sec1}. 
According to Duhamel's principle, our goal is to identify a unique fixed point of the map $\Phi:Y_\delta^s\to Y_\delta^s$ defined by
\begin{equation} \label{sol}
	\Phi(u)=e^{-(\partial_x^5+\alpha\partial_x^3)t}g-\frac{1}{2}\int_0^te^{-(\partial_x^5+\alpha\partial_x^3)(t-t')}\partial_x(u^2)dt'.
\end{equation}
We then plan to utilized the subsequent lemmas to demonstrate that $\Phi$ acts as a contraction map on
$$\mathcal S:=\left\{u\in Y_{\delta}^s:\|u\|_{Y_{\delta}^s}\le M\right\}.$$
\begin{lem}[\cite{ASS}]\label{lem2}
Let $s\in\mathbb R$. For any $\delta<1$, the following estimates hold:
	\begin{align}
		\label{w1}&\|e^{-(\partial_x^5+\alpha\partial_x^3)t}g\|_{Y_\delta^s}\lesssim\|g\|_{H^s},\\
		\label{w2}&\bigg\|\int_0^te^{-(\partial_x^5+\alpha\partial_x^3)(t-t')}F(\cdot,t')dt'\bigg\|_{Y_\delta^s}\lesssim\|F\|_{Z_\delta^s}.
	\end{align}
\end{lem}
\begin{lem}[\cite{ASS}]\label{lem}
Let $s\ge0$ and $-1/2<b<b'<1/2$.
For any $\delta<1$, the following estimates hold:
\begin{align*}
&\|u\|_{X_\delta^{s,b}}\lesssim_{b,b'}\delta^{b'-b}\|u\|_{X_\delta^{s,b'}},\\
&\|u\|_{Z_\delta^s} \lesssim\delta^{1-}\|u\|_{Y_\delta^s},\\
&\|\partial_x(u_1u_2)\|_{Z_\delta^s}\lesssim_{\alpha,s}\|u_1\|_{X_\delta^{s,\frac3{10}}}\|u_2\|_{X_\delta^{s,\frac3{10}}}
\end{align*}
where $u_1,u_2$ are mean-zero functions.
\end{lem}

We first show that $\Phi(u)\in Y_\delta^s$ for $u\in Y_\delta^s$.
For this, we apply \eqref{w1} and \eqref{w2} to the linear term and the Duhamel term in \eqref{sol} respectively to get
$$\|e^{-(\partial_x^5+\alpha\partial_x^3)t}g\|_{Y_\delta^s} \le C\|g\|_{H^s},$$
\begin{equation}\label{duh}
\left\| \int_0^te^{-(\partial_x^5+\alpha\partial_x^3)(t-t')}F(u)(\cdot,t')dt'\right\|_{Y_\delta^s}
\le C \|\partial_x(u^2)\|_{Z_\delta^s}.
\end{equation}
Using Lemma \ref{lem}, we bound the right-hand side of \eqref{duh} as
\begin{equation*}
C\|\partial_x(u^2)\|_{Z_\delta^s}
\le C \|u\|_{{X^{s,\frac3{10}}_\delta}}^2\le C \delta^{\frac{2}{5}-}\|u\|_{{X_\delta^{s,\frac{1}{2}-}}}^2.
\end{equation*}
Since $\|u\|_{X_\delta^{s,\frac{1}{2}-}}\le\|u\|_{X_\delta^{s,\frac12}}\le\|u\|_{Y_\delta^{s}}$,
we therefore get
$$\|\Phi(u)\|_{Y_\delta^s} \le C\|g\|_{H^s}+  C \delta^{\frac{2}{5}-} M^2.$$
If we fix $M=2C\|g\|_{H^s}$ and take $\delta$ sufficiently small such that
\begin{equation}\label{qs}
C\delta^{\frac{2}{5}-}M\le \frac{1}{2},
\end{equation}
then we get $\|\Phi(u)\|_{Y_\delta^s}\le M.$

Next we show that $\Phi$ is a contraction on $\mathcal S$. Again using Lemma \ref{lem}, we see that
\begin{align*}
\|\Phi (u)-\Phi (v)\|_{Y_\delta^s}&\le C \|\partial_x(u^2-v^2)\|_{Z_\delta^s}\\
&\le C\delta^{\frac{2}{5}-}\|u+v\|_{Y_\delta^{s}}\|u-v\|_{Y_\delta^{s}}\\
&\le C \delta^{\frac{2}{5}-}M\|u-v\|_{Y_\delta^s}.
\end{align*}
Hence $\Phi$ is a contraction on $\mathcal S$ since we are
taking $\delta$ so that \eqref{qs} holds.
Therefore, by the contraction mapping principle, there exists a unique solution
$u \in Y_{\delta}^s$ for given initial data $g \in H^s$ with $s\ge0$.
\end{proof}

\section{Smoothing estimates}\label{sec4}
This final section is dedicated to proving Proposition \ref{thm1}.
To do this, we first transcribe the equation \eqref{KA} in terms of the Fourier transform as presented in \eqref{fwdk3} and make the substitution $u_k(t)=w_k(t)e^{-i(k^5-\alpha k^3)t}$.
Next, we employ differentiation by parts to drag the phase down from the oscillation arising in this process.
This approach enables the utilization of decay of the resulting large denominators.
Moreover, while the degree of nonlinearity is elevated from quadratic to cubic,
we split the associated trilinear term into resonant and nonresonant components and employ Bourgain's restricted norm method to control the delicate nonresonant components effectively.

\subsection{Representation of Fourier coefficients of the solution}\label{sub4-1}
The lemma below provides a formulation for the solution to \eqref{KA} in terms of the Fourier transform.
\begin{lem}
The solution $u$ to \eqref{KA} is represented in the following form:
\begin{align}\label{morphback}
\nonumber
u_k(t)&=e^{-i(k^5-\alpha k^3)t}g_k+\mathcal B(u,u)_k(t)-e^{-i(k^5-\alpha k^3)t}\mathcal B(u,u)_k(0)\\
\nonumber
&\qquad+\int_0^te^{-i(k^5-\alpha k^3)(t-r)}\rho(u)_k
dr+\int_0^te^{-i(k^5-\alpha k^3)(t-r)} \sigma(u)_k
dr\\
&\qquad\qquad+\int_0^te^{-i(k^5-\alpha k^3)(t-r)}\mathcal R(u)_k(r)dr,
\end{align}
where for $k\ne0$
\begin{align*}
&\mathcal B(\phi,\psi)_k=-\frac12\sum_{k_1+k_2=k}\frac{\phi_{k_1}\psi_{k_2}}{k_1k_2\{5(k_1^2 +k_2^2 + k_1k_2)-3\alpha\}},\\ &\rho(u)_k=\frac{i|u_k|^2u_k}{2(15 k^2-3\alpha)k}, \quad \sigma(u)_k=-iu_k\sum_{|j|\ne|k|}\frac{|u_j|^2}{j\{5(k^2-kj+j^2)-3\alpha\}},\\
&\mathcal{R}(u)_k=-\frac i2\sum_{\substack{k_1+k_2+k_3=k\\(k_1+k_2)(k_2+k_3)(k_3+k_1)\ne0}}\frac{u_{k_1} u_{k_2} u_{k_3}} {k_1\{5(k_1^2 +(k_2+k_3)^2 + k_1(k_2+k_3))-3\alpha\}},
\end{align*}
and $B(\phi,\psi)_0=\rho(u)_0=\sigma(u)_0=R(u)_0=0$.
\end{lem}
\begin{proof}
Applying the Fourier transform to the equation \eqref{KA} results in
\begin{equation}\label{fwdk3}
\begin{cases}
\partial_tu_k=-i k^5u_k+ i\alpha k^3u_k-\frac{ik}2\sum_{k_1+k_2=k}u_{k_1}u_{k_2}, \\
u_k(0) = g_k.
\end{cases}
\end{equation}
Given the mean-zero condition of $u$ (refer to Section \ref{sec1}), there are no zero harmonics in this equation.
Using the substitution
$u_k(t)=w_k(t)e^{-i(k^5-\alpha k^3)t}$
and employing the identities
\begin{align*}
	k_1^3+k_2^3&=(k_1+k_2)^3-3k_1k_2(k_1+k_2),\\
	k_1^5+k_2^5&=(k_1+k_2)^5-5k_1k_2(k_1+k_2)(k_1^2+k_2^2+k_1k_2),
\end{align*}
we can reformulate the equation \eqref{fwdk3} as
\begin{equation}\label{morph}
	\partial_tw_k=-\frac{ik}2\sum_{k_1+k_2=k}e^{i\{5(k_1^2 + k_2^2 + k_1 k_2 )-3\alpha\} k_1k_2kt}w_{k_1}w_{k_2}.
\end{equation}

Using differentiation by parts with $e^{at}= \partial_t (\frac1a e^{at})$, we can reformulate \eqref{morph} as
\begin{align}\label{morph1}
\nonumber
\partial_tw_k
	&=-\partial_t\Big(\sum_{k_1+k_2=k}\frac{e^{ik_1k_2k\{5(k_1^2 + k_2^2 + k_1 k_2 )-3\alpha\}t}w_{k_1}w_{k_2}}{2k_1k_2\{5(k_1^2 + k_2^2 + k_1 k_2 )-3\alpha\}}\Big)\\
	&\qquad\qquad\qquad\qquad\quad+\sum_{k_1+k_2=k}\frac{e^{ik_1k_2k\{5(k_1^2 + k_2^2 + k_1 k_2 )-3\alpha\}t}w_{k_1}\partial_t w_{k_2}}{k_1k_2\{5(k_1^2 + k_2^2 + k_1 k_2 )-3\alpha\}}.
\end{align}
Since $w_0=0$, neither $k_1$ nor $k_2$ in the sums above are zero.
We now introduce $\widetilde{\mathcal{B}}$ defined as 
$$ \widetilde{\mathcal B}(\phi,\psi)_k=-\frac12\sum_{k_1+k_2=k}\frac{e^{ik_1k_2k\{5(k_1^2 + k_2^2 + k_1 k_2)-3\alpha\}t}\phi_{k_1}\psi_{k_2}}{k_1k_2\{5(k_1^2 + k_2^2 + k_1 k_2 )-3\alpha\}}.$$
Then we reformulate the equation \eqref{morph1} as follows:
\begin{equation}\label{morphing1}
\begin{aligned}
\partial_t(w-\widetilde{\mathcal{B}}(w,w))_k=\sum_{k_1+k_2=k}\frac{e^{ik_1k_2k\{5(k_1^2 + k_2^2 + k_1 k_2 )-3\alpha\}t}w_{k_1}\partial_t w_{k_2}}{k_1k_2\{5(k_1^2 + k_2^2 + k_1 k_2 )-3\alpha\}}.
\end{aligned}
\end{equation}

Inserting \eqref{morph} into the sum within \eqref{morphing1}, we observe
\begin{align}
\nonumber
&\sum_{k_1+k_2=k}\frac{e^{ik_1k_2k\{5(k_1^2 + k_2^2 + k_1 k_2 )-3\alpha\}t}w_{k_1}\partial_t w_{k_2}}{k_1k_2\{5(k_1^2 + k_2^2 + k_1 k_2 )-3\alpha\}}\\
\label{morphing2}
&\qquad\quad=\sum_{\substack{k_1+k_2+k_3=k\\ k_2+k_3\ne0}}\frac{e^{it\theta(k_1,k_2,k_3)(k_1+k_2)(k_2+k_3)(k_3+k_1)}}{2k_1\{5(k_1^2 + (k_2+k_3)^2 + k_1 (k_2+k_3) )-3\alpha\}}w_{k_1}w_{k_2} w_{k_3},
\end{align}
where ${\theta} (k_1,k_2,k_3)=5(k_1^2 +k_2^2+k_3^2+k_1k_2+k_2k_3+k_3k_1)-3\alpha$.
By merging \eqref{morphing1} with \eqref{morphing2}, we derive
\begin{align}
\nonumber
&\partial_t(w-\widetilde{\mathcal{B}}(w,w))_k\\
\label{morphing3}
&\qquad =-\frac i2\sum_{\substack{k_1+k_2+k_3=k\\ k_2+k_3\ne0}}\frac{e^{it \theta(k_1,k_2,k_3)(k_1+k_2)(k_2+k_3)(k_3+k_1)}}{k_1\{5(k_1^2 + (k_2+k_3)^2 + k_1 (k_2+k_3) )-3\alpha\}}w_{k_1}w_{k_2}w_{k_3}.
\end{align}
Given that $\theta (k_1, k_2, k_3 )\ne 0$, the zero-phase sets are individually the disjoint unions of:
\begin{align*}
	S_1 &= \{k_1+k_2=0, k_3 + k_1=0, k_2 +k_3 \ne 0\}
	=\{k_1=-k, k_2 = k, k_3 = k \}, \\
	S_2&= \{k_1 +k_2 = 0,k_3+k_1 \ne 0,k_2+k_3 \ne 0\}
	= \{k_1=j, k_2 = -j, k_3 = k, |j|\ne|k| \},\\
	S_3 &= \{k_3+k_1=0, k_1 + k_2 \ne 0, k_2 + k_3 \ne 0\}
	= \{k_1=j, k_2 = k, k_3 = -j, |j|\ne|k| \}.
\end{align*}
We define $\widetilde {\mathcal R}$ by
$$\widetilde {\mathcal R}(u)_k=-\frac i2\sum_{\substack{k_1+k_2+k_3=k\\(k_1+k_2)(k_2+k_3)(k_3+k_1)\ne0}}\frac{e^{it \theta(k_1,k_2,k_3)(k_1+k_2)(k_2+k_3)(k_3+k_1)}u_{k_1}u_{k_2}u_{k_3}} {k_1\{5(k_1^2 + (k_2+k_3)^2 + k_1 (k_2+k_3) )-3\alpha\}}.
$$
Thus, \eqref{morphing3} is reformulated as
\begin{align}
\nonumber
&\partial_t(w-\widetilde {\mathcal B}(w,w))_k\\
\nonumber
&\qquad\quad=-\frac{i}{2} \sum_{l=1}^3 \sum_{S_l} \frac{w_{k_1}w_{k_2}w_{k_3} }{k_1\{5(k_1^2 + (k_2+k_3)^2 + k_1 (k_2+k_3) )-3\alpha\}}+\widetilde {\mathcal R}(w)_k\\
\label{morphgoal}
&\qquad\quad=\frac{i|w_k|^2w_k}{2(15 k^2-3\alpha)k}-iw_k\!\sum_{|j|\ne|k|}\!\frac{|w_j|^2}{j\{5(k^2-kj+j^2)-3\alpha\}}+\widetilde {\mathcal R}(w)_k.
\end{align}
Upon integrating \eqref{morphgoal} from 0 to $t$, it follows
\begin{align*}
w_k(t)-w_k(0)&=\widetilde{\mathcal B}(w,w)_k(t)- \widetilde{\mathcal B}(w,w)_k(0)\\
&\quad +\int_0^t\rho(w)_k(r)dr+\int_0^t\sigma(w)_k(r)dr+\int_0^t\widetilde{\mathcal R}(w)_k(r)dr.
\end{align*}
Converting back to the $u$ variable, we conclude
\begin{align*}
u_k(t)-e^{-i(k^5-\alpha k^3)t}g_k&=\mathcal B(u,u)_k(t)-e^{-i(k^5-\alpha k^3)t}\mathcal B(u,u)_k(0)\\
&\qquad+\int_0^te^{-i(k^5-\alpha k^3)(t-r)}\big(\rho(u)_k(r)+\sigma(u)_k(r)\big) dr\\
&\qquad\qquad+\int_0^te^{-i(k^5-\alpha k^3)(t-r)}\mathcal R(u)_k(r)dr
\end{align*}
as desired.
\end{proof}

\subsection{Estimates for $\mathcal{B}, \rho, \sigma$ and $\mathcal{R}$}
We now derive the following multilinear estimates, Proposition \ref{goodlem}, which will be employed to complete the proof of the smoothing estimates (Proposition \ref{thm1}):
\begin{prop}\label{goodlem}
Let $s>-1/2$. For $s_1\le s+3$, we have
\begin{equation}\label{prop1}
	\|\mathcal B(\phi,\psi)\|_{H^{s_1}}\lesssim_{\alpha}\|\phi\|_{H^s}\|\psi\|_{H^s},
\end{equation}
and for $s_1\le s+2$, we have
\begin{equation}\label{prop2}
\|\rho(u)\|_{H^{s_1}}\lesssim_{\alpha}\|u\|_{H^s}^3, \quad
\|\sigma(u)\|_{H^{s_1}}\lesssim_{\alpha}\|u\|_{H^s}^3.
\end{equation}
For $0\leq s_1 <\min(s+2,3s+2)$ we have
\begin{equation}\label{prop3}
	\|\mathcal{R}(u)\|_{X_{\delta}^{s_1, -1/2+\epsilon}} \lesssim_{\alpha,\epsilon}\|u\|^3_{X_{\delta}^{s,1/2}}.
\end{equation}
\end{prop}
\subsubsection{Proof of \eqref{prop1}}
Due to symmetry, we can restrict the sum in $\mathcal B$ to $|k_1|\ge|k_2|$.
Since $s>-1/2$ and $$5(k_1^2+k_2^2+k_1k_2)-3\alpha\sim_{\alpha}\left(k_2 + \frac{k_1}{2}\right)^2 + \frac{3k_1^2}{4}\ge k_1^2\gtrsim k^2,$$
we calculate
\begin{align*}
\|\mathcal B(\phi,\psi)\|_{H^{s_1}}&\le\left\||k|^{s_1}\sum_{\substack{k_1+k_2=k\\|k_1|\ge|k_2|}}
	\frac{|\phi_{k_1}\psi_{k_2}|}{|5(k_1^2+k_2^2+k_1k_2)-3\alpha||k_1k_2|}\right\|_{l_k^2}\\
	&\lesssim_{\alpha}\left\| \sum_{\substack{k_1+k_2=k\\|k_1|\ge|k_2|}}
	\frac{|\phi_{k_1}||k_1|^s|\psi_{k_2}||k_2|^s|k|^{s_1-s-3}}{|k_2|^{s+1}}\right\|_{l_k^2}.
\end{align*}
For $s>-1/2$ and $s_1 \le s+3$, employing Young's inequality and the Cauchy-Schwarz inequality ensures that the right-hand side is bounded as follows:
\begin{align*}
\left\|\sum_{\substack{k_1+k_2=k\\|k_1|\ge|k_2|}}\frac{|\phi_{k_1}||k_1|^s|\psi_{k_2}||k_2|^s|k|^{s_1-s-3}}{|k_2|^{s+1}}\right\|_{l_k^2}
&\le\left\|\sum_{\substack{k_1+k_2=k\\|k_1|\ge|k_2|}}\frac{|\phi_{k_1}||k_1|^s|\psi_{k_2}||k_2|^s}{|k_2|^{s+1}}\right\|_{l_k^2}\\
&\le\||\phi_k||k|^s\|_{l_k^2}\bigg\|\frac{|\psi_k||k|^s}{|k|^{s+1}}\bigg\|_{l_k^1}\\
&\le\|\phi\|_{H^s}\big\||k|^{-s-1}\big\|_{l_k^2}\|\psi\|_{H^s}\lesssim\|\phi\|_{H^s}\|\psi\|_{H^s}
\end{align*}
yielding the required estimate.
\subsubsection{Proof of \eqref{prop2}}
For $s_1 \leq s+2$, employing the inclusion properties of $l^p$-spaces, we derive the first inequality in \eqref{prop2} as follows:
\begin{align*}
\|\rho(u)\|_{H^{s_1}}= \bigg\|\frac{|u_k|^3 |k|^{3s}|k|^{s_1-3s}}{2|15 k^2-3\alpha||k|}\bigg\|_{l_k^2}&\lesssim_{\alpha}\bigg\||u_k|^3 |k|^{3s}|k|^{s_1-3s-3}\bigg\|_{l_k^2}\\
&\leq\||u_k|^3|k|^{3s}\|_{l_k^2}=\||u_k||k|^s\|^3_{l_k^6} \le\|u\|_{H^s}^3.
\end{align*}

Regarding the second inequality in \eqref{prop2}, the left-hand side is bounded as follows:
\begin{align*}
\|\sigma(u)\|_{H^{s_1}}&\leq\bigg\||k|^{s_1-s}|u_k||k|^s\sum_{|j|\neq|k|}\frac{|u_j|^2|j|^{2s}|j|^{-2s}}{|j||5(k^2-kj+j^2)-3\alpha|}\bigg\|_{l_k^2} \\
&\le \sup_{k,j\ne0}\frac{|k|^{s_1-s}}{|j|^{2s+1}|5(k^2-kj+j^2)-3\alpha|}\|u\|_{H^s}\sum_{|j|\neq|k|}|u_j|^2|j|^{2s} \\
&\le \sup_{k,j\ne0}\frac{|k|^{s_1-s}}{|j|^{2s+1}|5(k^2-kj+j^2)-3\alpha|}\|u\|^3_{H^s}.
\end{align*}
We utilized the Cauchy-Schwarz inequality for the last inequality.
For $s\geq-1/2$ and $s_1\le s+2$,
$$\sup_{k,j\ne0}\frac{|k|^{s_1-s}}{|j|^{2s+1}|5(k^2-kj+j^2)-3\alpha|}\lesssim_{\alpha} \sup_{k,j\ne0}\frac{|k|^{s_1-s-2}}{|j|^{2s+1}}\lesssim1$$
since $k^2+j^2-kj=\frac12(k^2+j^2+(k-j)^2)>\frac12k^2$. This completes the proof.

\subsubsection{Proof of \eqref{prop3}}
To prove \eqref{prop3}, it is sufficient to consider a local-in-time function $u$.
Using duality, we demonstrate
\begin{align}\label{dualgoal}
\nonumber
\bigg|\int_{\mathbb T^2}\mathcal R(u)(x,t)&h(x,t)dxdt\bigg|\\
&=\bigg|\sum_{k,m}\widehat{\mathcal R}(k,m)\widehat h(-k,-m)\bigg|
\lesssim_{\alpha,\epsilon}\|u\|_{X^{s,\frac{1}{2}}}^3\|h\|_{X^{-s_1,\frac12-\epsilon}}.
\end{align}

First, we note that
\begin{align*}
&\widehat{\mathcal R}(k,m)\\&=\sum_{\substack{k_1+k_2+k_3=k\\(k_1+k_2)(k_2+k_3)(k_3+k_1)\ne0}}\sum_{m_1+m_2+m_3=m}\frac{-i\widehat u(k_1,m_1)\widehat u(k_2,m_2)\widehat u(k_3,m_3)}{2k_1\{5(k_1^2 +(k_2+k_3)^2 + k_1(k_2+k_3))-3\alpha\}}.
\end{align*}
We define
\begin{align*}
	&\Phi:=\{(k_1,k_2,k_3,k_4)\in \mathbb{Z}^4:k_1+k_2+k_3+k_4=0, (k_1+k_2)(k_2+k_3)(k_3+k_1)\ne0\},\\
	&\Omega:=\{(m_1,m_2,m_3,m_4)\in\mathbb{Z}^4:m_1+m_2+m_3+m_4=0\},
\end{align*}
and we define
\begin{align*}
f_i (k_i,m_i) &:=|\widehat u(k_i,m_i)||k_i|^s\langle m_i+k_i^5-\alpha k_i^3\rangle^{\frac{1}{2}},\quad i=1,2,3,\\
f_4(k,m)&:=|\widehat h(k_4,m_4)||k_4|^{-s_1}\langle m_4+k_4^5-\alpha k_4^3\rangle^{\frac12-\epsilon},\\
\langle m_n+k_n^5-\alpha k_n^3\rangle&=\max_{i=1,2,3,4}\langle m_i+k_i^5-\alpha k_i^3\rangle,\\ S&:=\{1,2,3,4\}\backslash\{n\}.
\end{align*}
Therefore, \eqref{dualgoal} is deduced from the inequality
\begin{align}\label{a}
\nonumber \sum_{\Phi,\Omega}&\frac{|k_1k_2k_3|^{-s}|k_4|^{s_1}\prod_{i=1}^4f_i(k_i,m_i)}{|k_1||5(k_1^2 +(k_2+k_3)^2 + k_1(k_2+k_3))-3\alpha|\prod_{i=1}^4\langle m_i+k_i^5-\alpha k_i^3\rangle^{\frac12-\epsilon}}\\
&\qquad \qquad \qquad \qquad \qquad \qquad \qquad \qquad \qquad\qquad\qquad\quad \qquad\lesssim_{\alpha,\epsilon}\prod_{i=1}^4\|f_i\|.
\end{align}

Now we proceed to prove \eqref{a}.
It is noted that
\begin{align}
\nonumber
|5(k_1^2+k_2^2+k_3^2+k_1k_2+k_2k_3&+k_3k_1)-3\alpha|\\
\nonumber
&\gtrsim_{\alpha}|k_1+k_2|^2 + |k_2+k_3|^2 + |k_3+k_1|^2 \\
\label{g}
&\gtrsim (|k_1+k_2||k_2+k_3||k_3+k_1|)^{\frac23}.
\end{align}

With \eqref{g} in place, utilizing the following two inequalities to be demonstrated later:
\begin{align}\label{j}
\nonumber
\prod_{i=1}^4\langle m_i+k_i^5&-\alpha k_i^3\rangle^{\frac12-\epsilon}\\
&\gtrsim 
(|k_1+k_2||k_2+k_3||k_3+k_1|)^{\frac{5}{6}-\frac{35}{3}\epsilon}\prod_{i\in S}\langle m_i+ k_i^5-\alpha k_i^3\rangle^{\frac12+\epsilon},
\end{align}
and for $0\leq s_1<\min(s+2,3s+2)$ (sufficiently small $\epsilon>0$)
\begin{align}\label{k}
\frac{|k_1k_2k_3|^{-s}|k_4|^{s_1}}{|k_1|(|k_1+k_2||k_2+k_3||k_3+k_1|)^{\frac{3}{2}-\frac{35}{3}\epsilon}}\lesssim|k_1k_2k_3k_4|^{-\epsilon},
\end{align}
then, the left-hand side of \eqref{a} is bounded by
\begin{align}
\label{dualgoal2}
\sum_{\Phi,\Omega}\frac{|k_1k_2k_3k_4|^{-\epsilon}\prod_{i=1}^4f_i(k_i,m_i)}{\prod_{i\in S}\langle m_i+k_i^5-\alpha k_i^3\rangle^{\frac12+\epsilon}}.
\end{align}
We may remove $|k_1|^{-\epsilon}$ from \eqref{dualgoal2} and apply the following lemma from \cite{ASS} to bound \eqref{dualgoal2}:
\begin{lem}\label{l6}
	For any $\epsilon>0$ and $b>1/2$, we have
	\begin{equation*}
		\|\chi_{[0,\delta]}(t)u\|_{L_{t,x}^6(\mathbb R\times\mathbb T)}\lesssim_{\alpha,\epsilon,b}\|u\|_{X_\delta^{\epsilon,b}}.
\end{equation*}
\end{lem}
Indeed, using the convolution structure, along with Plancherel's theorem and H"older's inequality as well as Lemma \ref{l6}, \eqref{dualgoal2} is controlled as
\begin{align*}
\sum_{\Phi,\Omega}&\frac{|k_2k_3k_4|^{-\epsilon}\prod_{i=1}^4f_i(m_i,k_i)}{\prod_{i\in S}\langle m_i+k_i^5-\alpha k_i^3\rangle^{\frac12+\epsilon}}\\
&\qquad\qquad \quad \le \int_{\mathbb T^2}\left|\check{f_1}(x,t)\prod_{i=2}^4\bigg(\frac{f_i|k_i|^{-\epsilon}}{\langle m_i+ k_i^5-\alpha k_i^3\rangle^{\frac12+\epsilon}}\bigg)^\vee(x,t)\right|dxdt\\
&\qquad\qquad\quad\le\|f_1\|_{L^2}\prod_{i=2}^4\bigg\|\bigg(\frac{f_i|k|^{-\epsilon}}{\langle m+k^5-\alpha k^3\rangle^{\frac12+\epsilon}}\bigg)^\vee\bigg\|_{L^6}\lesssim_{\alpha,\epsilon}\prod_{i=1}^4\|f_i\|_{L^2},
\end{align*}
as needed.

It remains to prove \eqref{j} and \eqref{k}. We begin by demonstrating \eqref{j}. Given that $m_1+m_2+m_3+m_4=0$ and $k_1+k_2+k_3+k_4=0$, it follows that
\begin{align*}
\sum_{i=1}^4&m_i+k_i^5-\alpha k_i^3\\
&=-(k_1+k_2)(k_2+k_3)(k_3+k_1)\{5(k_1^2+k_2^2+k_3^2+k_1k_2+k_2k_3+k_3k_1)-3\alpha\}.
\end{align*}
Consequently, by using \eqref{g},
\begin{align}
\nonumber
\langle m_n+ k_n^5-\alpha k_n^3\rangle \gtrsim \sum_{i=1}^4 |m_i+ k_i^5-\alpha k_i^3|&\ge \Big|\sum_{i=1}^4 m_i+ k_i^5-\alpha k_i^3\Big|\\ 
%&\,\,\,\,\,\ge|k_1+k_2||k_2+k_3||k_3+k_1||5\alpha(k_1^2+k_2^2+k_3^2+k_1k_2+k_2k_3+k_3k_1)-3\beta| \\
\label{e}
&\gtrsim (|k_1+k_2||k_2+k_3||k_3+k_1|)^{\frac53}.
\end{align}
Using \eqref{e}, we can then conclude that
\begin{align*}
\prod_{i=1}^4\langle &m_i+ k_i^5-\alpha k_i^3\rangle^{\frac12-\epsilon}\\
&=\langle m_n+ k_n^5-\alpha k_n^3\rangle^{\frac12-\epsilon}\prod_{i\in S}\langle m_i+ k_i^5-\alpha k_i^3\rangle^{-2\epsilon}\langle m_i+ k_i^5-\alpha k_i^3\rangle^{\frac12+\epsilon}\\
&\ge\langle m_n+ k_n^5-\alpha k_n^3\rangle^{\frac12-7\epsilon}\prod_{i\in S}\langle m_i+ k_i^5-\alpha k_i^3\rangle^{\frac12+\epsilon}\\
&\gtrsim_{\alpha}(|k_1+k_2||k_2+k_3||k_3+k_1|)^{\frac{5}{6}-\frac{35}{3}\epsilon}\prod_{i\in S}\langle m_i+k_i^5-\alpha k_i^3\rangle^{\frac12+\epsilon},
\end{align*}
thereby establishing \eqref{j}.

Finally, we establish \eqref{k}. We begin by noting that
\begin{equation}\label{dualwork2}
	|k_1+k_2||k_2+k_3||k_3+k_1|\gtrsim|k_i|,\quad i=1,2,3,4,
\end{equation}
because $(k_1+k_2)(k_2+k_3)(k_1+k_3)\neq0$ and $k_1+k_2+k_3+k_4=0$.
From \eqref{dualwork2}, the inequality \eqref{k} is derived from
\begin{equation*}
\frac{|k_1k_2k_3|^{-s}|k_4|^{s_1}}{|k_1|(|k_1+k_2||k_2+k_3||k_3+k_1|)^{\frac{3}{2}-\frac{47}{3}\epsilon}}\lesssim 1.
\end{equation*}
Using $|k_1||k_1+k_2|\gtrsim|k_2|$ and $|k_1||k_1+k_3|||k_2+k_3|\gtrsim|k_2|$, and by symmetry among $k_2$, $k_3$, we observe \begin{align*}
&|k_1|(|k_1+k_2||k_2+k_3||k_3+k_1|)^{\frac{3}{2}-\frac{47}{3} \epsilon}\\
&\,\,=|k_1|(|k_1+k_2||k_2+k_3||k_3+k_1|)^{\frac{1}{2}-\frac{47}{3} \epsilon}|k_1+k_2||k_2+k_3||k_3+k_1|\gtrsim M^{2-47\epsilon}
\end{align*}
where $M=\max\{|k_1|,|k_2|,|k_3|\}$.
Thus, we also find
\begin{equation}\label{m}
\frac{|k_1k_2k_3|^{-s}|k_4|^{s_1}}{|k_1|(|k_1+k_2||k_2+k_3||k_3+k_1|)^{\frac{3}{2}-\frac{47}{3}\epsilon}} \lesssim \frac{|k_1k_2k_3|^{-s}|k_4|^{s_1}}{M^{2-47\epsilon}}.
\end{equation}
Since 
\begin{equation*}
|k_1k_2k_3|^{-s} \leq
\begin{cases}
M^{-s} \quad \textrm{if}\quad s\ge0, \\ M^{-3s} \quad \textrm{if} \quad s\le0,	
\end{cases}	
\end{equation*}
it follows that $|k_1k_2k_3|^{-s} \leq M^{-\min(s,3s)}.$
Using this information, for $0\leq s_1 < \min(s+2,3s+2)$ and sufficiently small $\epsilon$, the right-hand side of \eqref{m} is bounded as
$$M^{-\min(s,3s)+s_1-2+47\epsilon}\lesssim 1.$$
This conclusion finalizes the proof.
\subsection{The last step of the proof}
We now conclude the proof of Proposition \ref{thm1}.
By applying the estimates from Proposition \ref{goodlem} to the equation \eqref{morphback}, 
we derive for $s>-\frac12$ and $0\leq s_1<\min(s+2,3s+2)$
\begin{align}
\nonumber
\|u(t)-e^{Lt}g\|_{H^{s_1}}&\lesssim_{\alpha}\|u(t)\|_{H^s}^2+\|g\|_{H^s}^2+\int_0^t\|u(r)\|_{H^s}^2dr+\int_0^t\|u(r)\|_{H^s}^3dr\\
\label{int}
&\qquad\quad+\bigg\|\int_0^te^{L(t-r)}\mathcal R(u)(r)dr\bigg\|_{H^{s_1}}.
\end{align}
Next, we address the final term in \eqref{int} by utilizing the embedding
$$X_{\delta}^{s,b}\subset L_t^\infty H_x^s([0,\delta]\times \mathbb{T}), \quad b>1/2,$$
and by referring to the following lemma from \cite{GTV}:
\begin{lem}\label{ri}
	Let $b>1/2$.
	Then for any $\delta<1$,
$$\bigg\|\int_0^te^{L(t-r)}F(r)dr\bigg\|_{X_\delta^{s,b}}\lesssim_{b}\|F\|_{X_\delta^{s,b-1}}.$$
\end{lem}
For $b>1/2$ and $t<\delta$ with $\delta$ given in Theorem \ref{thm2}, we see
\begin{align}\label{rinitiate3}
\nonumber
\bigg\|\int_0^t e^{L(t-r)}\mathcal R(u)(r)dr\bigg\|_{H^{s_1}}&\le\bigg\|\int_0^t e^{L(t-r)}\mathcal R(u)(r)dr\bigg\|_{L_{t\in[0,\delta]}^\infty H_x^{s_1}}\\
\nonumber
&\lesssim\bigg\|\int_0^te^{L(t-r)}\mathcal R(u)(r)dr\bigg\|_{X_\delta^{s_1,b}}\\
&\lesssim_{b}\|\mathcal R(u)\|_{X_\delta^{s_1,b-1}}\lesssim_{\alpha,b}\|u\|_{X_\delta^{s,1/2}}^3.
\end{align}
Here, for the last inequality, we used \eqref{prop3}.

Combining \eqref{int} with \eqref{rinitiate3}, it is observed that for $t<\delta$,
\begin{align*}
\|u(t)-e^{Lt}g\|_{H^{s_1}}\lesssim\|u(t)\|_{H^s}^2+\|g\|_{H^s}^2+\int_0^t\big(\|u(r)\|_{H^s}^2+\|u(r)\|_{H^s}^3\big)dr+\|u\|_{X_\delta^{s,1/2}}^3.
\end{align*}
The implicit constants later in the proof are dependent on $|g|{H^s}.$
Next, fix a large $t$.
For $r\le t$, it is held that $|u(r)|{H^s}\lesssim T(r) \leq T(t)$.
Here, we assume, without loss of generality, that $T(t)\ge1$.
From the local estimate \eqref{bd}, for any $j\le t/\delta$, it follows
$$\|u\|_{X_{[(j-1)\delta,j\delta]}^{s,1/2}}\lesssim\|u((j-1)\delta)\|_{H^s}\lesssim T(t).$$
As a result, choosing $\delta \approx T(t)^{-\frac52}$ ensures
$$\|u(j\delta)-e^{\delta L}u((j-1)\delta)\|_{H^{s_1}}\lesssim T(t)^3.$$
Applying this, with $J=t/\delta$, we deduce
\begin{align*}
\|u(J\delta)-e^{J\delta L}g\|_{H^{s_1}}&\le\sum_{j=1}^J\|e^{(J-j)\delta L}u(j\delta)-e^{(J-j+1)\delta L}u((j-1)\delta)\|_{H^{s_1}}\\
&=\sum_{j=1}^J\|u(j\delta)-e^{\delta L}u((j-1)\delta)\|_{H^{s_1}}\\
&\lesssim JT(t)^3 \approx t T(t)^{\frac{11}{2}}.
\end{align*}
This completes the proof. Following similar steps as outlined above, one can easily demonstrate the continuity of $u(t)-e^{Lt}g$ in $H^{s_1}$.

\end{document}